\documentclass[11pt]{amsart}
\usepackage{epsfig}
\usepackage{graphicx}
\usepackage{color}

\newtheorem{theorem}{Theorem}[section]

\newtheorem{proposition}[theorem]{Proposition}

\theoremstyle{definition}

\DeclareMathOperator{\tr}{tr}

\theoremstyle{remark}
\newtheorem{remark}[theorem]{Remark}

\numberwithin{equation}{section}

\title[Kauffman bracket versus Jones polynomial]{Kauffman bracket versus Jones polynomial skein modules}
\author{Shamon Almeida}
\address{Department of Mathematics and Statistics, Texas Tech University, Lubbock, TX 79409}
\email{shamon.almeida@gmail.com}
\author{R{\u{a}}zvan Gelca}
\address{Department of Mathematics and Statistics, 
Texas Tech University, Lubbock, TX 79409}
\email{rgelca@gmail.com}



\begin{document}

\begin{abstract}
  This paper resolves the problem of comparing the skein modules
  defined using the skein relations discovered
  by P.~Melvin and R.~Kirby  that underlie the quantum group based Reshetikhin-Turaev
  model for $SU(2)$ Chern-Simons theory
  to the Kauffman bracket skein modules.
  Several applications and examples are presented.
\end{abstract}

\maketitle

\section{Motivation}

In 1984 V.F.R.~Jones introduced a polynomial invariant of knots and links  \cite{jones}. Immediately
after, L.~Kauffman defined a similar polynomial knot and link invariant, the Kauffman bracket,
which is in fact an invariant of framed knots and links \cite{kauffman}. 
Kauffman has shown how the Jones polynomial of a knot
can be computed from the Kauffman bracket.

  In 1989 E.~Witten in \cite{witten1989} has explained the Jones polynomial 
        by means of  a quantum field theory based on the Chern-Simons functional. The Jones polynomial corresponds to the particular case of the
         Chern-Simons theory with gauge group $SU(2)$. By making
        use of physical intuition, Witten predicted the Jones polynomial to
 be part of a more general family of knot, link,  and manifold invariants.
 Motivated by Witten's ideas, Reshetikhin and Turaev constructed the  knot, link,
 and manifold invariants of the $SU(2)$ Chern-Simons theory
 using a quantum group associated to $sl(2,{\mathbb C})$ \cite{reshetikhinturaev}. This theory fulfills Witten's predictions. 
An analogous theory was developed for
the Kauffman bracket by Blanchet, Habegger, Masbaum, and Vogel
in \cite{bhmv}, and this theory parallels that of
Reshetikhin and Turaev. Each of these two parallel theories have lead to
significant developments and the aim of the present paper is to
explain the relationship between the two models at the most general level.

 Within the Reshetikhin-Turaev theory, and already present
in previous works by Reshetikhin himself, lies the Jones  polynomial of framed knots and links, but with a slightly different normalization.
This polynomial fits exactly the quantum field theoretical model from Witten's paper, it is the
polynomial that  Chern-Simons theory would associate to a link whose components are
colored by the $2$-dimensional irreducible representation of  $SU(2)$.
We will refer to this polynomial as the Jones polynomial in the 
  Reshetikhin-Turaev normalization (or simply the Jones polynomial,
  when there is no possibility of confusion). 
The coloring of a knot by the $n$-dimensional 
irreducible representation of $SU(2)$ yields a polynomial invariant of knots called the colored Jones polynomial, of which the Jones polynomial in the Reshetikhin Turaev normalization itself corresponds to $n=2$.
The convention is that the $n$th colored Jones polynomial of a knot $K$,
denoted by $J(K,n)$,   corresponds
to the coloring of $K$ by the $n+1$st irreducible representation.

\begin{figure}[h]
  \centering
   \scalebox{.4}{\input{crossing.pstex_t}}

\caption{}
\label{crossing}
\end{figure}

Let $M$ be a compact, orientable, 3-dimensional manifold $M$, on
which an orientation has been chosen. 
A framed link in
$M$ is an embedding
of finitely many annuli.

Let us discuss first the case $M=S^3$. Both the Kauffman bracket and the Jones
polynomial in the Reshetikhin-Turaev normalization of a framed knot or link
in $S^3$ can be computed using skein relations, and these skein relations
are quite similar. We denote the Kauffman
bracket of a link $L$ by $\left<L\right>$ and this version of the
Jones polynomial by $J_L$,
both in the variable $t$. 
 To write down the skein relations, let $L,H,V$ be
three framed links that coincide except in a ball where they are
as shown in Figure~\ref{crossing}. What this means is that
we have an orientation preserving embedding of the ball into $S^3$ such
that the preimage of the three links through this embedding look
as depicted in the diagrams. The Kauffman bracket has
the skein relations
\begin{eqnarray*}
\left<L\right>=t\left<H\right>+t^{-1}\left<V\right>, \quad \left<O\right>=-t^2-t^{-2}.
\end{eqnarray*}
Here and below $O$ is the unknot. The first equality expresses the
relation between the brackets of $L, H,$ and $V$, while the second expresses
the fact that every trivial link component can be replaced by multiplication
by the scalar $-t^2-t^{-2}$.

On the other hand, the skein relations of the
Jones polynomial in the Reshetikhin-Turaev normalization
have been computed by R.~Kirby
and P.~Melvin in \cite{kirbymelvin}; they  are
\begin{eqnarray*}
J_L=tJ_H+t^{-1}J_V\mbox{ or } J_L=\epsilon(tJ_H-t^{-1}J_V), \quad J_O=t^2+t^{-2}.
\end{eqnarray*}
There are two skein relations for resolving a crossing,
the one on the left is used when different link
components cross, meaning that the two crossing strands in the diagram $L$
from Figure~\ref{crossing} come from
different link components, and the skein relation on the right is used when
the diagram $L$ corresponds to a link component crossing itself, with $\epsilon$ being the sign of the crossing.
To compute $\epsilon$, one chooses any of the two possible orientations
of the link component, which then orients the two strands inside
the ball, and then the sign is computed using the right hand rule.
Both orientations of the link yield the same value for $\epsilon$. 

A great amount of Chern-Simons theory is dedicated to the study of the
combinatorial properties of knots and links decorated by irreducible representations of quantum groups (the so called quantized Wilson lines), and the
algebraic topological concept that lies at the heart of this study
is that of a skein module. 
Following J.~Przytycki \cite{przytycki},
we construct the skein module of a compact, orientable,
3-dimensional manifold $M$ on which an orientation has been fixed.
We do this by considering first the free
${\mathbb C}[t,t^{-1}]$-module with basis the isotopy classes of framed
links in $M$, and then factoring this module by the skein relations. In the case
of the Kauffman bracket we obtain the Kauffman bracket skein module
$K_t(M)$, obtained by factoring the above mentioned free module
by the submodule spanned by the elements of the form $L-tH-t^{-1}V$, where
$L, H, V$ are framed links that coincide except in a ball that is
embedded by an orientation preserving homeomorphism in which
they look as in Figure~\ref{crossing}, and also by the relation
that states that every link that contains a trivial link component (one that bounds a disk so that the framing is inside the disk) is
equivalent to the same link with that component erased,  multiplied by
$-t^2-t^{-2}$.

For the Jones polynomial, the  skein module
of $M$ was defined in \cite{gelcauribe}; it is denoted by $RT_t(M)$ to point
out that it comes from  the Reshetihin-Turaev theory.
It is  defined like for  the Kauffman bracket,
but with the Kirby-Melvin skein relations instead. We should point out that
the choice of the orientation of the manifold $M$ determines uniquely the
sign of the self-crossing of a link component, exactly like in the case
of $S^3$, and that this sign can be computed by choosing either of the
two orientations of the link component and then using the right hand rule
in the embedded ball. In \cite{gelcauribe} it was explained how
several constructs of $SU(2)$ Chern-Simons theory can be reduced
to these skein modules. 

For a better understanding of the need to introduce the skein modules
of the Reshetikhin-Turaev theory, let us contrast the two skein relations
in the so called ``classical case''. 
When $t=-1$, the Kauffman bracket skein relation yields the trace identity
for the negative of the trace of $sl(2,{\mathbb C})$ characters of
the fundamental group of $M$:
\begin{eqnarray*}
 ( -\tr\rho(\alpha\beta))+(-\tr\rho(\alpha))(-\tr\rho(\beta))+
(- \tr\rho(\alpha\beta^{-1}))=0, 
\end{eqnarray*}
as it has been noticed in \cite{bullock}. On the other hand, the skein
relation of Kirby and Melvin yields, when $t=1$, the trace identity
for the trace itself
\begin{eqnarray*}
 \tr\rho(\alpha\beta)-\tr\rho(\alpha)\tr\rho(\beta)+
\tr\rho(\alpha\beta^{-1})=0. 
\end{eqnarray*}
In Chern-Simons theory $t=e^{i\pi h}$, where $h$ is interpreted, depending
on the context, as either
the coupling constant or Planck's constant. 
Setting $t=1$ is equivalent to setting the coupling constant or Planck's constant equal to zero, and this is predicted to correspond to
the classical (nonquantized) situation, that is to the character
variety. This physical interpretation, and the fact that it
is more natural to work with the trace than the negative of the trace,
are one of the reasons that we have  proposed in our previous work  the study of the skein modules $RT_t(M)$
of the Jones polynomial. The other reason is that the fundamental facts of $SU(2)$ Chern-Simons theory (the
Murakami Theorem \cite{murakami}, the Volume Conjecture \cite{murakami2}, the AJ Conjecture \cite{garle}) are phrased in the quantum group setting.
But other constructs (such as the quantum Teichm\"uller theory \cite{bonahonwong}, \cite{fkl}) are phrased in the  Kauffman bracket setting,
and  the present  paper  clarifies  the relationship
between the two types of skein theories: $K_t(M)$ and $RT_t(M)$.

\section{The main result}

Let $M$  be a compact, orientable 3-dimensional manifold on which an
orientation has been chosen, and let
$L$ be a framed link in $M$. Consider a compact orientable 3-dimensional
manifold $N$ such that $\partial N=-\partial M$, and
consider the closed manifold $M\cup N$ obtained by gluing $M$ and
$N$ along their common boundary.

The 3-dimensional manifold $M\cup N$, being closed, can be obtained from $S^3$ by performing surgery along
a framed link $L'$. Without loss of generality we may assume
that the solid tori of the surgery along $L'$ are disjoint from $L$. As such, $M\cup N$
is the boundary of a 4-dimensional manifold $W'$ obtained by gluing 2-handles to the 4-dimensional
ball $B^4$. Let us further glue 2-handles   to $W'$ along the components of the framed link $L$ to obtain
a 4-dimensional manifold $W$. Note that $W$ is obtained by gluing
$2$-handles to $B^4$ as specified by the framed links $L$ and $L'$ (both of which can now be viewed
as embedded annuli in $S^3$), so that 2-dimensional
disks are glued along the actual link components. These 2-handles define  homology classes in $H_2(W,{\mathbb Z})$,
which homology classes are determined by the closed surfaces obtained by capping each disk by a
Seifert surface in $S^3$ of the corresponding link component. Now let us focus only on the homology classes
classes in $H_2(W,{\mathbb Z})$ determined by the link components of $L$, and let us denote by $\tr(L)$ the
trace of the intersection matrix of these homology classes.
This
trace is the sum of $[L_j]\cdot [L_j]$ over the components $L_j$ of
$L$, where $[L_j]\cdot [L_j]$ is the algebraic intersection number
of the homology class $[L_j]$ defined by $L_j$ with itself.

In earnest, the
intersection form on $H_2(W, {\mathbb Z})$ depends on an additional piece
of information: the orientation of the surfaces that are being intersected.
When restricted to the homology classes that arise from the link components
of $L$, that additional piece of information is encoded in an orientation
of the link components. But the elements on the diagonal of this
matrix do not depend on the orientation, they compute self-crossings, and
so $\tr(L)$ is well defined, and can be computed by choosing any such
orientation.

Note that $\tr(L)$ depends on the choice of $N$ and $W$,
but this fact does not alter the conclusion of the following theorem,
and in practical applications one should  always make the simplest choice.

Additionally, for a link $L$, we denote by $n(L)$ the number of components of $L$.

\begin{theorem}\label{maintheorem}
  The equality
  \begin{eqnarray*}
\sum_{k=1}^n c_kL_k=0
  \end{eqnarray*}
holds in $K_t(M)$  for some Laurent polynomials $c_k\in {\mathbb C}[t,t^{-1}]$ and some framed links
  $L_k$ in $M$ if and only if the  equality
  \begin{eqnarray*}
\sum_{k=1}^n(-1)^{n(L_k)+\tr(L_k)}c_kL_k=0
  \end{eqnarray*}
  holds in $RT_t(M)$.
\end{theorem}

\begin{proof}
  Note that the write down these formulas we have to implicitly choose
  some orientations of the the link components, but the
  formulas themselves, and the proof below, do not depend on these
  orientations. 
  All we have to show is that the statement of the theorem
  is invariant under skein relations. We have to examine three
  cases. 

  {\em Case 1.} If two components $L_\alpha$ and $L_\beta$ of one of the links
  $L_k$ cross, then after resolving the crossing the number of components
  dropped by $1$. On the diagonal of the intersection matrix the
  entries $[L_\alpha]\cdot [L_\alpha]$ and  $[L_\beta]\cdot [L_\beta]$
  disappear, and the entry
  \begin{eqnarray*}
    [L_\alpha]\cdot [L_\alpha]+ [L_\beta]\cdot [L_\beta]+2[L_\alpha]\cdot [L_\beta]-1
  \end{eqnarray*}
  appears,   thus the exponent $n(L_k)+\tr(L_k)$ changes by an even number.
  And indeed, the skein relation for two disjoint components that cross
  is the same for the Kauffman bracket and for the Jones polynomial in the
  Reshetikhin-Turaev normalization.

  {\em Case 2.} If a  component $L_\alpha$ of some link  $L_k$ crosses itself,
  the crossing can be positive or negative. Let $H$ and $V$ be
  the diagrams obtained after resolving the crossing. If the crossing is
  positive, then  $\tr(V)=\tr(H)=\tr(L_k)-1$, and $V$ has the same
  number of components as $L_k$ while the $H$ term has one component more.
  Thus when passing from the Kauffman bracket to the Jones polynomial
  we keep the same sign in front of $H$, while we change the sign
  in front of $V$, exactly as in the skein relation for the Jones polynomial
  in the Reshetikhin-Turaev normalization. If the crossing
  is negative, then $\tr(V)=\tr(H)=\tr(L_k)+1$, but this time
  the number of components stays the same in $H$ and increases in $V$.
  And this is again consistent with the skein relation.

  {\em Case 3.} If we remove a trivial component, then the Kauffman
  bracket is multiplied by $-t^2-t^{-2}$, while the Jones polynomial
  is multiplied by $t^2+t^{-2}$. In this case the number of link components
  decreases by $1$, and so the exponent of $-1$ decreases by $1$, changing
  the sign of the corresponding term. The theorem is proved. 
    \end{proof}

  If we vary $N$ and $W$ we just multiply by a $\pm 1$ the
 entire  second equation from the statement. 

\begin{remark}
  If $M\subset S^3$, we can choose $N=\overline{S^3\backslash M}$ and let $W$
  be the $4$-dimensional ball that $S^3$ bounds. Then the intersection
  matrix whose trace is $\tr(L)$ is just the linking matrix of $L$.

  Note that you can swap the two relations in order to pass from
  $RT_t(M)$ to $K_t(M)$. 
\end{remark}

If we work over the field of fractions ${\mathbb C}(t)$, we obtain
the following immediate corollary.

\begin{proposition}
The vector spaces $K_t(M)$ and $RT_t(M)$ are isomorphic.
\end{proposition}

\begin{proof}
  As isotopy classes of framed links span $K_t(M)$,
  we can find a basis consisting of
  framed links. But then this basis is a spanning set for $RT_t(M)$. It
  is either a basis, or it contains a basis. If it is not a basis,
  then the basis it contains is a spanning set for $K_t(M)$, a contradiction.
  Thus any basis of framed links of $K_t(M)$ is a basis of framed links of $RT_t(M)$.
  Hence the conclusion. 
\end{proof}

\section{Applications and examples}

Let us introduce the polynomials 
$T_n(\xi)=2\cos [n\arccos (\xi/2)]$, $n\in {\mathbb Z}$,  which is  a normalized version of
the Chebyshev polynomial polynomial of the first kind, and
$S_n(\xi)=\sin [(n+1)\arccos(\xi/2)]/\sin \arccos(\xi/2)$, $n\in {\mathbb Z}$,
which is a normalized  version
  of the Chebyshev polynomial of the second kind. 
  For a framed knot $K$ in some compact, oriented, 3-dimensional manifold $M$
  and a positive integer $m$, we let $K^m$ be the framed
  link consisting of $m$ parallel copies of $K$, where in order to
  produce the parallel copies $K$ is pushed in the direction of the framing.
  Given a framed link  $L=L_1\cup L_2\cup \cdots\cup  L_k$   and a $k$-tuple of positive integers, $(j_1,j_2,...,j_k)$, 
  we can construct the link $L_1^{j_1}\cup L_2^{j_2}\cup\cdots \cup L_k^{j_k}$
  by taking parallel copies of each component.

  In particular, for a knot $K$ we can construct the skeins $T_n(K)$
  and $S_n(K)$ in either $K_t(M)$ or $RT_t(M)$. 

  \subsection{The product-to-sum formula and Weyl quantization}
Let us give another reason for our focus on the skein modules
of the Jones polynomial in the Reshetikhin-Turaev normalization.
If a manifold is a cylinder over a surface,
then the operation of gluing one cylinder on top of the other induces an
algebra structure on the skein module; this is the skein algebra of the surface.
Of particular interest is the skein algebra of the 
the torus,  $K_t({\mathbb T}^2)$. As a module, it is free with basis
$(p,q)_T$, $p,q\in {\mathbb Z}$, $p\geq 0$, where $(p,q)_T=T_n((p/n,q/n))$,
with $n$  the greatest common divisor of $p$ and $q$ and
$(p/n,q/n)$  the curve of slope $q/p$ on the torus whose framing is parallel to
the torus. 

As shown in \cite{frohmangelca} and \cite{gelcauribe}, for both the Kauffman bracket and the Jones polynomial in the Reshetikhin-Turaev normalization,
the multiplication is
given by the product-to-sum formula
\begin{eqnarray*}
(p,q)_T(r,s)_T=t^{ps-qr}(p+r,q+s)_T+t^{-ps+qr}(p-r,q-s)_T.
\end{eqnarray*}
  
For a manifold with boundary, the operation of gluing a cylinder over the boundary to the 3-dimensional manifold induces a module structure on
its skein module, over the
skein algebra of the boundary. A situation that was investigated in \cite{frohmangelca} and \cite{gelcauribe2} is that where the manifold is the solid torus.
Let $\alpha$ be the curve that is the core of the solid torus (the image
of $(1,0)$ under the inclusion of the boundary). The following result was
proved in \cite{gelcauribe2} and \cite{gelcauribe}. 

\begin{proposition}\label{weylformula}
  In the case of the Jones polynomial in the Reshetikhin-Turaev normalization, the action
  of the skein algebra of the cylinder over the
  torus on the skein module of the solid torus is given by
\begin{eqnarray}\label{weyl1}
  (p,q)_TS_{j-1}(\alpha)=t^{-pq}[t^{2jq}S_{j-p-1}(\alpha)+t^{-2jq}S_{j+p-1}(\alpha)],
\end{eqnarray}
\end{proposition}

A  consequence of Theorem~\ref{maintheorem} is the following.

\begin{proposition} For the Kauffman bracket, the action of the skein algebra
  of the cylinder over the torus on the skein module of the solid torus is
  given by 
\begin{eqnarray*}
  (p,q)_TS_{j-1}(\alpha)=(-1)^qt^{-pq}[t^{2jq}S_{j-p-1}(\alpha)+t^{-2jq}S_{j+p-1}(\alpha)].
\end{eqnarray*}
\end{proposition}

\begin{proof}
Let $p=np'$, and $q=nq'$, with $p',q'$ coprime. Then
$(p,q)_T$ is a linear combination of links, each of which having
the number of components congruent to $n$ modulo $2$. Each of these components
is a copy of the curve  of slope $q/p$ on the torus.
Because we work with the blackboard framing of the torus, each
component contributes $(p'-1)q'+q'=p'q'$ to $\tr(L)$, and so modulo $2$, each
term of $(p,q)_T$ contributes $np'q'$ to the trace.
And $S_{j-1}(\alpha)$ contributes nothing to the exponent of -1 in the
formula from Theorem~\ref{maintheorem}. Also,
modulo $2$, the number of link components in  $S_k(\alpha)$ is $k$. Thus
when switching from the Jones polynomial picture to the Kauffman bracket
picture, 
(\ref{weyl1}) becomes
\begin{eqnarray*}
  (-1)^{np'q'+n+j-1}  (p,q)_TS_{j-1}(\alpha)=t^{-pq}[t^{2jq}(-1)^{j-np'-1}S_{j-p-1}(\alpha)\\
    +(-1)^{j+np'-1}t^{-2jq}S_{j+p-1}(\alpha)],
\end{eqnarray*}
An easy case check shows that if $p',q'$ are coprime then
$p'q'+1-p'\equiv q'(\mbox{mod }2)$, so $np'q'+n-np'\equiv nq'(\mbox{mod }2)$ and
the formula is proved.
\end{proof}

We point out that for a curve $\gamma$, in
the setting of the Reshetikhin-Turaev theory
$S_j(\gamma)$ corresponds to $\gamma$ colored by the $j+1$-dimensional
irreducible representation, $V^{j+1}$, of the corresponding quantum group
 (which we denote by $V^{j+1}(\gamma)$),
while in the setting of
the Kauffman bracket it corresponds
to the coloring of the curve by the $j$th Jones-Wenzl idempotent.
So the relation from Proposition~\ref{weylformula} has the nicer form
\begin{eqnarray*}
 (p,q)_TV^j(\alpha)=t^{-pq}[t^{2jq}V^{j-p}(\alpha)+t^{-2jq}V^{j+p}(\alpha)].
  \end{eqnarray*}
This equation has been related by the second author and
A.~Uribe \cite{gelcauribe2} to the action of the Heisenberg group on 
theta functions discovered by A.~Weil \cite{weil}, and as such to
the Weyl quantization of the moduli space of $SU(2)$ connections
on the tours, and this gives a second reason for our focus on the skein
modules of the Jones polynomial.
More explicitly, the moduli space in question is the ``pillow case'' obtained
by factoring the complex plane ${\mathbb C}$ by the maps $z\mapsto z+m+ni$,
$m,n\in {\mathbb Z}$ and $z\mapsto -z$. To perform geometric quantization
we let Planck's constant be the reciprocal of an even integer $h=(2r)^{-1}$,
and let $\zeta_j$ be the sections of the Chern-Simons line
bundle over the moduli space that are lifted to the plane as the entire functions as
\begin{eqnarray*}
  \zeta_j=\sqrt[4]{r}e^{-\frac{j^2\pi}{2r}}(\theta_j-\theta_{j}), \quad \theta_j(z)=
  \sum_{n=-\infty}^\infty e^{-\pi(2rn^2+2jn)+2\pi i z(j+2rn)}.
\end{eqnarray*}
Then we let $C(p,q)$ be the operator associated by peforming equivariant
Weyl quantization to the function $2\cos (2\pi(px+qy))$ on the pillow case
(here $z=x+iy$). A computation with integrals yields
\begin{eqnarray*}
C(p,q)\zeta_j=t^{-pq}[t^{2jq}\zeta_{j-p}+t^{-2jq}\zeta_{j+p}], \mbox{ where } t=e^{\frac{i\pi}{2r}},
\end{eqnarray*}
which has been interpreted as saying that the Weyl quantization and the
quantum group quantization of the moduli space of flat $SU(2)$ connections
on the torus coincide. So this result  puts the accent on the use of
the skein modules $RT_t(M)$.

\subsection{The skein module of the complement of the $(2p+1,2)$ torus
  knot}
Let us now show an example that arises in the search for patterns in
skein modules. Computations with skeins have exponential complexity,
and these computations are expected to yield complicated results. Sometimes, for
apparently no reason, the result of a lengthy computation produces a
simple formula. This is the case with the following
example, which we will examine, for
comparison, in both situations.
The Kauffman bracket skein module of the complement  $S^3\backslash N(T_{2p+1,2})$ of
a regular neighborhood of the  $(2p+1,2)$ torus knot $T_{2p+1,2}$ is free
with basis $x^ny^k$, $n\geq 0$, $0\leq k\leq p$, as it was shown
by D.~Bullock in \cite{bullock2},
where  $x$ and $y$  are  depicted in Figure~\ref{torusknot} and are
endowed with the blackboard framing.
  \begin{figure}[h]
\centering
\scalebox{.4}{\input{torusknot.pstex_t}}

\caption{}
\label{torusknot}
    \end{figure}
 Then $RT_t(S^3\backslash T_{2p+1,2})$  is also free, with the same basis. 
 Indeed, using Theorem~\ref{maintheorem} and Bullock's result
 we conclude that every skein
        in $RT_t(S^3\backslash T_{2p+1,2})$ is a linear combination of the elements
        $x^ny^k$, $n\geq 0, 0\leq k\leq p$. And any nontrivial linear combination
        equal to $0$ in $RT_t(S^3\backslash T_{2p+1,2})$ would yield a nontrivial linear
        combination equal to zero in $K_t(S^3\backslash T_{2p+1,2})$, which is impossible. 
  
  For the Kauffman bracket skein module, the following  surprising
  formula was discovered
  by J.~Sain in \cite{gelcasain2}
   \begin{eqnarray*}
   t^{-2i-1}S_{p+i}(y)+t^{2i+1}S_{p-i-1}(y)=(-1)^iS_{2i}(x)(t^{-1}S_p(y)+tS_{p-1}(y)),
   \end{eqnarray*}
   for $i=1,2,...,p+1$, 
   which allows the reduction of higher ``powers''  of $y$ to lower powers.
   By Theorem~\ref{maintheorem}, in the quantum group setting of the Reshetikhin-Turaev theory we
	have the slightly simpler identity
	\begin{eqnarray*}
	t^{-2i-1}V^{p+i+1}(y)-t^{2i+1}V^{p-i}(y)=V^{2i+1}(x)(t^{-1}V^{p+1}(y)-tV^p(y)).
	\end{eqnarray*}
	

    \subsection{The colored Jones polynomials and the noncommutative
    A-polynomial of a knot}
  If $K\subset S^3$ is a framed knot with framing zero, then, in $RT_t(S^3)$,
	the skein   $S_n(K)$ is equal to the colored Jones polynomial of $K$ corresponding to the
  coloring of $K$ by the $n+1$st irreducible representation of the
  the quantum group of $SU(2)$ multiplied by the empty link:
  \begin{eqnarray*}
S_n(K)=J(K,n)\emptyset.
    \end{eqnarray*}
  Theorem~\ref{maintheorem} shows that if
  we evaluate $S_n(K)$ in the Kauffman bracket skein module $K_t(S^3)$ instead, we obtain
  $(-1)^{n}J(K,n)\emptyset$, because the trace of each term of $S_n(K)$
  is zero and the number of componets is congruent to $n$ modulo 2. In
  other words, the $n$th colored Jones polynomial is equal to $(-1)^{n}$
  times the $n$th colored Kauffman bracket:
  \begin{eqnarray*}
J(K,n)=(-1)^{n}\left<S_{n}(K)\right>
    \end{eqnarray*}
  a fact that is being used widely (see for example \cite{lepaper}). 

  There are two versions of the definition of the noncommutative
  generalization of the A-polynomial of a knot and the aim of this
  paragraph is to give a better understanding of the relationship between
  the two. The first was defined by the second author in joint work
  with Ch.~Frohman and W.~Lofaro in \cite{frgelo} and is based on the
  Kauffman bracket. The construction uses the action of the Kauffman
  bracket skein algebra of the cylinder over the torus, $K_t({\mathbb T}^2)$,
    on the Kauffman bracket skein module $K_t(S^3\backslash(N(K))$
    of the   complement
    of the regular neighborhood of a knot $K$, which arises from
    gluing the cylinder to the knot complement. The annihilator
    of the empty link, which is a left ideal in $K_t({\mathbb T}^2)$, is
    called the peripheral ideal of the knot and is denoted by ${\mathcal I}_t(K)$. It consists of the linear
    combinations of framed
    curves on the boundary torus that become equal to zero when ``pushed''
    inside the skein module of the knot complement. If we extend this
    ideal to a left ideal in the ring $${\mathbb C}_t[l,l^{-1},m,m^{-1}]={\mathbb C}\left<l,l^{-1},m,m^{-1}\right>/(lm=t^2ml)$$
    using the inclusion of  $K_t({\mathbb T}^2)$ into this latter ring defined
    by $(1,0)\mapsto l+l^{-1}$, $(0,1)\mapsto (0,1)$ (see \cite{frohmangelca}), then restrict it
    to ${\mathbb C}_t[l,m]$, we obtain what is called the non-commutative A-ideal of $K$ \cite{frgelo}. 
		The reason for the definition is that for $t=-1$ this ideal is principal, and modulo a normalization,
		it is generated by the A-polynomial defined in \cite{ccgls}.
                Moreover, it has been observed in \cite{frgelo} and \cite{gelca} 
		that every element in the non-commutative A-ideal yields a recursive relation for
		the colored Kauffman brackets $<S_n(K)>=(-1)^nJ(K,n)$. 

  The second construction of the noncommutative generalization of the
  A-polynomial has its origin in \cite{garle} and is based on quantum groups,
  being therefore related to the Jones polynomial in the Reshetikhin-Turaev
  normalization. The idea is to view the family of colored Jones polynomials
  as a function $f:{\mathbb Z}\rightarrow {\mathbb C}[t,t^{-1}]$, $f(n)=J(K,n)$
  and consider the operators $L$ and $M$ on such functions
  $Lf(n)=f(n+1)$ and $Mf(n)=t^{2n}f(n)$. These operators satisfy $LM=t^{2}ML$,
  and so they generate the ring
  \begin{eqnarray*}
 {\mathbb C}_t[L,L^{-1},M,M^{-1}]={\mathbb C}\left<L,L^{-1},M,M^{-1}\right>/(LM=t^2ML).
    \end{eqnarray*}
  The recurrence ideal of the knot $K$ is the left ideal consisting of the
  polynomials $P(L,M)$ satisfying $P(L,M)f=0$, where $f$ is the function defined
  above. It has been shown in \cite{garle} that this ideal is always nonzero.
  The two constructions are related because any recursive relation
  for $<S_n(K)>=(-1)^nJ(K,n)$ can be transformed into a recursive relation
  for $J(K,n)$, so every element in the peripheral ideal defined in
  \cite{frgelo} can be transformed into
  an element in the recurrence ideal, but this transformation is somewhat ad hoc 
  because it requires several sign adjustments.

  However, if we use for the definition of the noncommutative A-ideal
	the skein modules of the Jones polynomial, thus
  working instead with the action of $RT_t({\mathbb T}^2)$
  on $RT_t(S^3\backslash N(K))$, then  the ideal resulting
  from extending the peripheral ideal to ${\mathbb C}_t[l,l^{-1},m,m^{-1}]$
  and then restricting to ${\mathbb C}_t[l,m]$ is automatically
  included in the recurrence ideal under the identification
  $l=L,m=M$; no more change of signs. 
  Moreover, Theorem~\ref{maintheorem} implies that
  to pass from the peripheral ideal for the case of the Kauffman bracket
  to that for the case of the Jones polynomial in the Reshetikhin-Turaev normalization, one has to
  substitute each $(p,q)_T$ by $(-1)^p(p,q)_T$.

\subsection{The skein module of the  complement of the figure-eight knot}
We  illustrate  the facts that we have just  discussed 
with the example of the figure-eight knot. Let therefore
 $K_8$ be the figure-eight knot and let
 $N(K_8)$ be one of its open regular neighborhoods. 
Consider the left action of the skein algebra of the torus,
$RT_t(\mathbb{T}^2)$, on $RT_t(S^3\backslash N(K_8))$ 
defined by gluing the cylinder over the torus to the boundary of the knot
complement such that the curve $(1,0)$ is identified with the longitude  and the curve $(0,1)$ is identified with the meridian of the knot. 
	To understand the $RT_t(\mathbb{T}^2)$-module structure of $RT_t(S^3\backslash N(K_8))$, we need to explicate the action
	of the elements $(p,q)_T$ from the boundary.

        The Kauffman bracket skein module
	of the figure-eight knot complement was found by D.~Bullock and W.~Lofaro
	in \cite{bullocklofaro} to be free  with basis $x^n,x^ny,x^ny^2$ where $n \geq 0$, or equivalently $x^n,x^ny,x^nz$, where $n \geq 0$, the framed curves
        $x,y$ and $z$ being shown in Figure~\ref{figureeight} and being
        endowed with the blackboard framing.
        For the same reason as in the case of the torus knots
        discussed above,  
        $RT_t(S^3\backslash N(K_8))$ is also free, with the same basis.
	\begin{figure}[h]
		\centering
		\scalebox{.4}{\input{figureeight.pstex_t}}
		
		\caption{}
		\label{figureeight}
	\end{figure}


      From the work in \cite{gelcasain} one can infer that the
        action of the algebra $K_t({\mathbb T}^2)$ on $K_t(S^2\backslash N(K_8))$ is
        determined by the following 
  \begin{eqnarray*}
           &&  (1,q)_T\emptyset=t^{q}[(t^2S_{2+q}(x)+t^{-2}S_{-q}(x))Y+(t^{2}S_{q}(x)+
              t^{-2}S_{2-q}(x))Z\\&& \quad -(t^2S_{-4-q}(x)+t^{-2}S_{-4+q}(x))],\\
    && (1,q)_TY=t^{q+4}[-(t^2S_{4+q}(x)+t^{-2}S_{-4-q}(x))Y-(t^2S_{q+2}(x)\\
      &&\quad +t^{-2}S_{-q}(x))Z  +(t^2S_{-6-q}(x)+t^{-2}S_q(x))],\\
    && (1,q)_TZ=t^{q-4}[-(t^{2}S_q(x)+t^{-2}S_{2-q}(x))Y-(t^2S_{-4+q}(x)\\
      &&\quad +t^{-2}S_{4-q}(x))Z +t^{q-4}(t^{-2}S_{-6+q}(x)+t^{2}S_{-q}(x))].
        \end{eqnarray*}
  where $Y=t^2y+1$, $Z=t^{-2}z+1$ and $q\in {\mathbb Z}$.
  We point out that the action of $K_t({\mathbb T}^2)$ on elements of the form
  $x^n$, as well as on $x^ny$ and $x^nz$ for $n>0$ can be derived from these
  using the product-to-sum formula.  This also only explicates the action of elements of the form
$(1,q)_T$, the product-to-sum formula allows the computation of the action
of a general $(p,q)_T$, albeit without a nice closed form formula. 

Applying Theorem~\ref{maintheorem}, and noticing that the computation of the signs requires
 just the counting link components modulo 2,
we obtain the following result.  
  
\begin{theorem}
  The   action of $RT_t({\mathbb T}^2)$ on $RT_t(S^2\backslash N(K_8))$ is
        determined by
          \begin{eqnarray*}
         &&   (1,q)_T\emptyset=t^q[(t^2S_{2+q}(x)+t^{-2}S_{-q}(x))Y+(t^{2}S_{q}(x)+
              t^{-2}S_{2-q}(x))Z\\
              &&\quad +(t^2S_{-4-q}(x)+t^{-2}S_{-4+q}(x))],\\
            &&  (1,q)_TY=t^{q+4}[(t^2S_{4+q}(x)+t^{-2}S_{-4-q}(x))Y+(t^2S_{q+2}(x)\\ &&\quad +t^{-2}S_{-q}(x))Z  +(t^2S_{-6-q}(x)+t^{-2}S_q(x))],\\
            &&  (1,q)_TZ=t^{q-4}[(t^{2}S_q(x)+t^{-2}S_{2-q}(x))Y+(t^2S_{-4+q}(x)\\
              &&\quad +t^{-2}S_{4-q}(x))Z +t^{q-4}(t^{-2}S_{-6+q}(x)+t^{2}S_{-q}(x))].
          \end{eqnarray*}
          where $Y=t^2y-1$, $Z=t^{-2}z-1$ and $q\in {\mathbb Z}$. 
        \end{theorem}

Using this module structure, after a tedious computation,  one  obtains the following example of an element
in the peripheral ideal of $K_8$ in the version that uses the Jones polynomial in
the Reshetikhin-Turaev normalization: 
  \begin{eqnarray*}
  &&t^{-6}(2,3)_T-t^6(2,-1)_T-t^3(1,7)_T+t(1,5)_T+(t^{11}-t^3+t^{-1}+t^{-5})(1,3)_T\\
    && +(-t^9+t^5+t^{-7})(1,1)_T+(t^{11}-2t^7-t^3+t^{-1}-t^{-9})(1,-1)_T\\
    &&+(-t^{13}-t)(1,-3)_T+t^{-1}(1,-5)_T+t^8(0,7)_T+(-2t^8+t^4-t^{-4})(0,5)_T\\
    &&+(-t^{12}+t^8-t^4-1+t^{-4})(0,3)_T+(t^{12}-t^8+1+t^{-4})(0,1)_T.
  \end{eqnarray*}
This should be contrasted with 
   \begin{eqnarray*}
  &&t^{-6}(2,3)_T-t^6(2,-1)_T+t^3(1,7)_T-t(1,5)_T+(-t^{11}+t^3-t^{-1}-t^{-5})(1,3)_T\\
     && +(t^9-t^5-t^{-7})(1,1)_T+(-t^{11}+2t^7+t^3-t^{-1}+t^{-9})(1,-1)_T\\
     &&+(t^{13}+t)(1,-3)_T-t^{-1}(1,-5)_T+t^8(0,7)_T+(-2t^8+t^4-t^{-4})(0,5)_T\\
     &&+(-t^{12}+t^8-t^4-1+t^{-4})(0,3)_T
  +(t^{12}-t^8+1+t^{-4})(0,1)_T 
   \end{eqnarray*}
   which was obtained in  \cite{gelcasain} as  an element of
   the peripheral ideal defined using the Kauffman bracket.
   The former gives rise to the following  recursive relation for colored
   Jones polynomials $y_n=J(K_8,n)$ of the figure-eight knot
\begin{eqnarray*}
  &&(t^{6n+6}-t^{-2n+2}) y_{n+2}+(-t^{14n+24}+t^{10n+16}+t^{6n+20}-t^{6n+12}+t^{6n+8}\\
  &&+t^{6n+4}
  -t^{2n+12}+t^{2n+8}
  +t^{2n-4}+t^{-2n+8}-2t^{-2n+4}-t^{-2n}+t^{-2n-4}\\
  &&-t^{-2n-12}-t^{-6n+4}-t^{-6n-8}+t^{-10n-16})y_{n+1}
  +(t^{14n+22}-2t^{10n+18}\\&&+t^{10n+14}
  -t^{10n+6}-t^{6n+18}+t^{6n+14}-t^{6n+10}-t^{6n+6}+t^{6n+2}
  +t^{2n+14}\\
  &&-t^{2n+10}+t^{2n+2}+t^{2n-2}+t^{-2n+10}-t^{-2n+6}+t^{-2n-2}+t^{-2n-6}-t^{-6n+6}\\
  &&+t^{-6n+2}-t^{-6n-2}-t^{-6n-6}+t^{-6n-10}-2t^{-10n-2}+t^{-10n-6}-t^{-10n-14}\\
  &&+t^{-14n-6})y_n
  +(t^{10n+4}-t^{6n+16}-t^{6n+4}+t^{2n+12}-2t^{2n+8}-t^{2n+4}+t^{2n}\\&&-t^{2n-8}
  -t^{-2n+8}+t^{-2n+4}
  +t^{-2n-8}+t^{-6n+8}-t^{-6n}+t^{-6n-4}+t^{-6n-8}\\&&+t^{-10-4}
  -t^{-14n-4})y_{n-1}
  +(t^{2n+6}+t^{-6n-6})y_{n-2}=0.
\end{eqnarray*}
If we use the construction based on the Kauffman bracket, we obtain
 a recursive relation for $(-1)^nJ(K,n)$ instead.

\end{document}